\documentclass[a4paper,11pt]{amsart}

\pdfoutput=1

\usepackage[T1]{fontenc}
\usepackage{amssymb,amsfonts,amsmath,amsthm}
\usepackage{url}
\usepackage{color}
\usepackage{enumerate}
\usepackage{pinlabel}

\usepackage{mathabx}

\usepackage[top=1.2in,bottom=1.4in,left=1.4in,right=1.4in]{geometry}

\input{xy}
\xyoption{all}
\objectmargin={3mm}

\newcounter{notes}
\newcommand{\ignore}[1]{}

\newtheorem{theorem}{Theorem}

\newtheorem{lemma}[theorem]{Lemma}

\theoremstyle{definition}

\newtheoremstyle{theoremwithref}{}{}{\itshape}{}{\bfseries}{.}{.5em}{#1 #2 #3}
\theoremstyle{theoremwithref}

\newcommand{\C}{\mathbb{C}}
\newcommand{\R}{\mathbb{R}}

\newcommand{\SL}{\mathrm{SL}}

\newcommand{\Hit}{\mathrm{Hit}}
\newcommand{\HIT}{\mathrm{HIT}}
\newcommand{\Hom}{\mathrm{Hom}}

\title{Most Hitchin representations are strongly dense}

\author{D. D.  Long}
\address{Department of Mathematics\\
University of California\\ Santa Barbara, CA 93106, USA.}
\email{long@math.ucsb.edu}

\author{A. W. Reid}
\address{Department of Mathematics\\
Rice University \\ Houston, TX 77005, USA.}
\address{Max-Planck-Insititut f\"ur Mathematik\\
Vivatsgasse 7, D-53111 Bonn, Germany}
\email{alan.reid@rice.edu, areid@mpim-bonn.mpg.de}

\author{M. Wolff}
\address{Sorbonne Universit\'es, UPMC Univ.\ Paris 06, Institut de Math\'ematiques
de Jussieu-Paris Rive Gauche, UMR 7586, CNRS, Univ. Paris Diderot, Sorbonne
Paris Cit\'e, 75005 Paris, France}
\email{maxime.wolff@imj-prg.fr}

\begin{document}

\maketitle
\numberwithin{theorem}{section}

\begin{abstract}
 We prove that generic Hitchin
  representations are strongly dense: every pair of
  non commuting elements in their image generate a
  Zariski-dense subgroup of $\SL_n(\R)$. The proof uses 
  a theorem of Rapinchuk, Benyash-Krivetz and
  Chernousov, to show that the set of Hitchin
  representations is Zariski-dense in the variety of
  representations of a surface group in $\SL_n(\R)$.
\vspace{0.2cm}

\end{abstract}

\sloppy
\section{Introduction}
\label{intro}

Following Breuillard, Green, Guralnick and Tao
\cite{BGGT}, we say that
a subgroup $\Gamma\subset\SL_n(\R)$ is {\em strongly dense} if any
pair of non-commuting elements of $\Gamma$ generate a
Zariski-dense subgroup of $\SL_n(\R)$. They proved that, among
many other semisimple algebraic groups, the group $\SL_n(\R)$
contains a strongly dense non abelian free
subgroup~\cite[Theorem 4.5]{BGGT}. In this note, we extend the  Breuillard, Green, Guralnick and Tao result to
certain (discrete and) faithful representations of surface groups of genus at least two into $\SL_n(\R)$.

To describe this more carefully, we introduce some background and terminology.  For fixed $g\geqslant 2$,  
and base field $k$, the set of representations of the surface group $\pi_1(\Sigma_g)$ to $\SL_n(k)$ is denoted by $\Hom(\pi_1(\Sigma_g), \SL_n(k))$ and is naturally an affine subvariety
of $k^{2g\cdot n^2}$ known as the {\em representation variety}. In the case of $k=\R$, those
representations of interest to us, the Hitchin representations, are of particular geometric importance and can be
defined as follows.

The {\em Teichm\"uller representations} in $\Hom(\pi_1(\Sigma_g), \SL_n(\R))$
are those obtained by composing any faithful and discrete representation
$\pi_1(\Sigma_g)\to\SL_n(\R)$ with an irreducible representation
$\SL_2(\R)\to\SL_n(\R)$. The Hitchin representations are those that lie in the same  
(topological) connected component of $\Hom(\pi_1(\Sigma_g),\SL_n(\R))$ as a Teichm\"uller representation. 
Note that, depending on the parity of $n$, there may be more than one such component, but we simply choose one and denote it by $\HIT_n$.
\footnote{We note that a {\em Hitchin component} more usually refers to a connected component
of the {\em character variety} $X(\pi_1(\Sigma_g),\SL_n(\R)))$ and the notation $\Hit_n$ is frequently used, but in this note
it will be technically simpler to work at the level of representations.}


We say that a representation is strongly dense if its image is a strongly dense subgroup of $\SL_n(\R)$, and we say that
a subset of $\Hom(\pi_1(\Sigma_g),\SL_n(\R))$ is {\em generic} if its complement consists of a countable union of proper subvarieties. The main result of this note is:
\begin{theorem}\label{thm:Generic-Hitchin-SD}
  Let $g\geqslant 2$ and $n\geqslant 3$. Then the set of strongly dense
  representations of $\pi_1(\Sigma_g)$ is generic in $\HIT_n$.
\end{theorem}
It is known that all the representations in $\HIT_n$ are faithful and discrete
(see \cite[Theorem 1.5]{Labourie}), 
so this provides the representations promised in the
first paragraph. We note that 
the result  of Theorem \ref{thm:Generic-Hitchin-SD} was obtained recently in ~\cite{LongReidSD} in the
case of $n=3$ by direct geometric methods.\\[\baselineskip]
To prove Theorem \ref{thm:Generic-Hitchin-SD} we prove the following result, which seems independently interesting, and uses a result of Rapinchuk,
Benyash-Krivetz and Chernousov \cite{Rapinchuk-al}, that $\Hom(\pi_1(\Sigma_g), \SL_n(\C))$ is an irreducible subvariety of $\C^{2g\cdot n^2}$; in fact, it is Zariski and classically connected.
\begin{theorem}\label{thm:Hitchin-Z-dense}
  For all $n\geqslant 2$, the set $\HIT_n$ is Zariski-dense in  the affine algebraic set $\Hom(\pi_1(\Sigma_g), \SL_n(\C))$.
\end{theorem}
\noindent The case $n=2$ was already essentially observed in \cite[Chapter 3]{Goldman80}.\\[\baselineskip]
As we describe below, Theorem \ref{thm:Generic-Hitchin-SD}  follows from Theorem \ref{thm:Hitchin-Z-dense} together with \cite{BGGT} and the fact that surface groups are residually free \cite{Baumslag}. \\[\baselineskip]
{\bf Acknowledgements}.  The authors wish to thank Bill Goldman,
Eran Iton, Fanny Kassel, Julien March\'e, Andr\'es Sambarino, 
and Nicolas Tholozan for encouragement and helpful conversations. The second author gratefully acknowledges the financial support of the N.S.F.  and the Max-Planck-Institut f\"ur Mathematik, Bonn, for its financial support and hospitality during the preparation of this work.

\section{Proofs.}

\noindent The proof of Theorem~\ref{thm:Hitchin-Z-dense} rests on the following observation:

\medskip

\noindent {\em Suppose that $V$ is an irreducible complex affine variety defined by
real polynomials. If $V$ has a smooth real point, then $V(\R)$ is Zariski dense in $V$.}

\medskip

Although apparently well-known and natural, we could not locate a proof in the literature, and so we decided to include a proof for completeness.

\begin{lemma}\label{lem:MerciJulien}
  Let $N\geqslant 1$; let $V(\C) \subset \C^N$ be an irreducible
  affine variety, defined over $\R$, of dimension $n$.
  Let $H$ be an open subset of $V(\R)$,
  for the usual, Hausdorff topology of $\R^N$,
  and suppose $H$ contains a smooth point $x_0$ of $V(\C)$.
  Then $H$ is Zariski-dense in $V(\R)$, for its structure as a
  real affine variety in $\R^N$.
\end{lemma}
\begin{proof}
  In fact we will prove that $H$ is Zariski-dense in $V(\C)$:
  any (complex) polynomial function on $\C^N$ vanishing
  identically on $H$, vanishes on $V(\C)$. By restricting
  to real polynomials, this implies the lemma.
  
  Let $F\colon\C^N\to\C^{N-n}$ be a polynomial map, with real
  coefficients, defining $V$. Since $x_0$ is a smooth point
  of $V(\C)$, or equivalently, of $V(\R)$, the Jacobian
  matrix $JF(x_0)$ has rank $N-n$. Up to permuting the
  coordinates, we may suppose that the minor of $JF(x_0)$
  corresponding to the $N-n$ last coordinates is invertible.
  Then the inverse function theorem yields a
  neighborhood $U\times V$ of $x_0$, and a map
  $\phi\colon U\to V$ with $U\subset \C^n$ and $V\subset \C^{N-n}$,
  such that $V(\C)\cap (U\times V)=F^{-1}(\{0\})\cap (U\times V)$
  is the graph of $\phi$.
  
  This map $\phi$ is obtained as the limit of an explicit fixed point
  process which is defined only in terms of $F$: it 
  follows that $\phi$ is holomorphic, as a uniform limit of holomorphic
  maps, and that the restriction of $\phi$ to the reals is the solution
  to the same inverse function problem, hence (up to taking smaller
  neighborhoods) for all $t\in U$ we have $t\in\R^n$ if and only if
  $\phi(t)\in\R^{N-n}$. See \textsl{e.g.} \cite[Paragraph 1.3]{Narasimhan}.
  
  Now let $P\colon\C^N\to\C$ be a polynomial function vanishing
  identically on $H$. Then the map $U\to\C$, $t\mapsto P(t,\phi(t))$
  is holomorphic, and it vanishes identically on $U\cap\R^n$.
  Such a map vanishes identically on $U$: this can be checked by
  induction on $n$, where both the base step and the induction
  step use the principle of isolated zeroes of holomorphic maps
  of one variable.
\end{proof}

\begin{proof}[Proof of Theorem~\ref{thm:Hitchin-Z-dense}]
As noted in \S \ref{intro},  $R(\C) = \Hom(\pi_1(\Sigma_g), \SL_n(\C))$
  is an affine subvariety of $\C^{2g\cdot n^2}$, and it was proved in \cite[Theorem 3]{Rapinchuk-al} to be irreducible of dimension $(2g-1)(n^2-1)$.
  
  The set $\HIT_n$ is, by definition, a (topological) connected component of $X_\R$, which is a real algebraic variety, and hence
  $\HIT_n$ is open.  We claim that it contains smooth points of $R(\R)$, or equivalently, of $R(\C)$: in fact, we will show that all its points are regular.
  
  Indeed, by a result of Goldman \cite[Propositon 1.2]{Goldman84},
  at each point $\rho$ of $R(\R)$, the dimension of the
  Zariski tangent space at $\rho$ equals  $(2g-1)(n^2-1)+\dim(\zeta(\rho(\pi_1\Sigma_g)))$, where
  $\zeta(\rho(\pi_1\Sigma_g))$ is the centralizer of the image
  group $\rho(\pi_1\Sigma_g)$ in $\SL_n(\R)$. 
  
  We will make use of the following facts proved by Labourie
  (see \cite[Theorem 1.5 and Paragraph 10]{Labourie}). First,  
  if $\rho\in \HIT_n$, then $\rho$ is irreducible, and second, for all nonidentity elements
  $\gamma\in\pi_1(\Sigma_g)$,
  the matrix $\rho(\gamma)$ is diagonalizable with pairwise
  distinct real eigenvalues.
  
  Fix such a $\gamma_0$; by conjugating the image of $\rho$ in $\SL_n(\R)$, we may suppose that $\rho(\gamma_0)$ is
  diagonal. Let $\xi$ be an element of $\zeta(\rho(\pi_1(\Sigma_g)))$. Since $\xi$ commutes with $\rho(\gamma_0)$, it is also
  diagonal, and if $\lambda$ is an eigenvalue of $\xi$,
  the matrix $\xi-\lambda I$ also commutes with  $\rho(\pi_1(\Sigma_g))$. Hence $\ker(\xi-\lambda I)$ is
  invariant by $\rho(\pi_1(\Sigma_g))$. However, $\rho$ is irreducible, and so
  this implies that $\xi$ is a scalar matrix, that is to say,  $\xi=\pm I$.
 
  Thus, the Zariski tangent space
  at any representation $\rho\in\HIT_n$ has minimal dimension,
  $(2g-1)(n^2-1)$, in other words, these are regular points of
  the varieties~$R(\R)$ and~$R(\C)$.
  
  The result now follows from Lemma~\ref{lem:MerciJulien}.
\end{proof}

\begin{proof}[Proof of Theorem \ref{thm:Generic-Hitchin-SD}]
  For every pair of non commuting elements $a, b\in \pi_1(\Sigma_g)$, let $\mathrm{Bad}(a,b)$
  denote the subset of $\Hom(\pi_1(\Sigma_g), \SL_n(\R))$
  consisting of representations $\rho$ such that $\rho(a)$ and
  $\rho(b)$ do not generate a Zariski-dense subgroup of $\SL_n(\R)$,
  and let $\mathrm{Good}(a,b)$ denote its complement.
  
  The proof will be complete once we know that for every pair of
  non commuting elements $a, b\in \pi_1(\Sigma_g)$, the set
  $\mathrm{Bad}(a,b)\cap \HIT_n$ is Zariski-closed, and that it is
  a proper subset of~$\HIT_n$.
  
  The fact that the sets $\mathrm{Bad}(a,b)$ are Zariski-closed follows
  from \cite[Theorem 4.1]{BGGT}.
  
  Now let us check that $\mathrm{Bad}(a,b)\cap \HIT_n$ is a proper
  subset of $\HIT_n$, or equivalently, that $\mathrm{Good}(a,b)\cap \HIT_n$
  is nonempty. Since $\mathrm{Good}(a,b)$ is Zariski-open, and since
  $\HIT_n$ is Zariski-dense in $\Hom(\pi_1(\Sigma_g), \SL_n(\R))$
  by Theorem~\ref{thm:Hitchin-Z-dense}, it suffices to check that
  $\mathrm{Good}(a,b)$ is nonempty.
  
  By \cite[Theorem 4.5]{BGGT}, there exists a strongly dense representation
  $\rho_0\colon F_2\to\SL_n(\R)$.
  Let $a, b\in \pi_1(\Sigma_g)$ be a pair of non commuting elements.
  Since $\pi_1(\Sigma_g)$ is
  residually free (see Baumslag~\cite{Baumslag}) and $[a,b]\neq 1$,
  there exists a surjective morphism $\psi$ from $\pi_1(\Sigma_g)$ onto a free
  group $F$, such that $\psi([a,b])\neq 1$. By composing $\psi$ with an injective
  morphism $F\to F_2$, this yields a morphism
  $\varphi\colon\pi_1(\Sigma_g)\to F_2$ such that $\varphi([a,b])\neq 1$.
  Thus, $\varphi(a)$ and $\varphi(b)$ do not commute, hence
  $\rho_0(\varphi(a))$ and $\rho_0(\varphi(b))$
  generate a Zariski dense subgroup of $\SL_n(\R)$.
  In other words, $\rho_0\circ\varphi$ lies in $\mathrm{Good}(a,b)$, so this
  set is non empty.
\end{proof}

\bibliographystyle{plain}

\bibliography{Biblio}

\end{document}